\documentclass{amsart}

\usepackage{amsmath,amssymb}
\usepackage{amsthm}
\usepackage{indentfirst}
\usepackage{hyperref}
\usepackage{xcolor}
\usepackage[margin=1.5in]{geometry}

\numberwithin{equation}{section}

\theoremstyle{plain}
\newtheorem{theorem}{Theorem}[section]

\newtheorem{prop}[theorem]{Proposition}

\newtheorem{corollary}[theorem]{Corollary}

\theoremstyle{definition}

\theoremstyle{remark}
\newtheorem*{remark}{Remark}

\theoremstyle{remarks}

\newcommand{\R}{\mathbb{R}}

\IfFileExists{mathabx.sty}%
  {\DeclareFontFamily{U}{mathx}{\hyphenchar\font45}%
   \DeclareFontShape{U}{mathx}{m}{n}{<->mathx10}{}%
   \DeclareSymbolFont{mathx}{U}{mathx}{m}{n}%
   \DeclareFontSubstitution{U}{mathx}{m}{n}%
   \DeclareMathAccent{\widebar}{0}{mathx}{"73}%
}{%
  \PackageWarning{mathabx}{%
    Package mathabx not available, therefore\MessageBreak substituting
    widebar with overline\MessageBreak }%
  \newcommand{\widebar}[1]{\overline{#1}}%
}

\newcommand{\la}{\lambda}
\newcommand{\e}{\varepsilon}

\AtBeginDocument{%
   \def\MR#1{}
}

\begin{document}

\author{Yannick Sire}
\address{Department of Mathematics\\                Johns Hopkins University\\                Baltimore, MD 21218}
\email{ysire1@jhu.edu}

\author{Christopher D. Sogge}
\address{Department of Mathematics\\                Johns Hopkins University\\                Baltimore, MD 21218}
\email{sogge@jhu.edu}

\author{Chengbo Wang}
\address{School of Mathematical Sciences\\                Zhejiang University\\                Hangzhou 310027, China}
\email{wangcbo@zju.edu.cn}

\author{Junyong Zhang}
\address{Department of Mathematics\\  Beijing Institute of Technology\\ Beijing 100081, China}
\email{zhang\_junyong@bit.edu.cn}

\title[Reversed Strichartz estimates]{Reversed Strichartz estimates for wave on non-trapping asymptotically hyperbolic manifolds and applications}

\begin{abstract}
We provide reversed Strichartz estimates for the shifted wave equations  on non-trapping asymptotically hyperbolic manifolds using cluster estimates for spectral projectors proved previously in such generality. 
As a consequence, we solve a problem left open in \cite{SSWZ} about the endpoint case for global well-posedness of nonlinear wave equations. We also provide estimates in this context for the maximal wave operator. 
\end{abstract}

\maketitle
\tableofcontents

\section{Introduction and main results}

The goal of this note is to obtain new reversed Strichartz estimates for the (shifted) wave operator on some complete manifolds with bounded geometry. We are concerned  in particular with asymptotically hyperbolic manifolds. The argument builds on $L^p$-estimates for the spectral projectors (on the continuous spectrum). We draw also two consequences of the reverse Strichartz estimates: first, the endpoint version of {\sl classical} Strichartz estimates and an application to global well-posedness for nonlinear wave equations on those manifolds; second, a sharp (in $L^p$ spaces) estimate of the maximal function for the wave operator.

We work on an $n+1$-dimensional complete non-compact Riemannian manifold $(M^\circ, g)$ where $n\geq 1$ and the metric $g$ is an asymptotically hyperbolic metric. 
This setting is the same as in Chen-Hassell
\cite{CH1,chenHassell}, Mazzeo \cite{M} and Mazzeo-Melrose\cite{MM}.  Let $x$ be a boundary defining function for the compactification $M$ of $M^\circ$.  We say a metric $g$ is conformally compact if $x^2g$ is
a Riemannian metric and extends smoothly up to the boundary $\partial M$. Mazzeo \cite{M} showed that its sectional curvature tends to $-|dx|^2_{x^2g}$ as $x\to0$; In particular, if the limit is such that $-|dx|^2_{x^2g}=-1$, we
say  that the conformally compact metric $g$ is asymptotically hyperbolic. More specifically, let $y=(y_1,\cdots,y_{n})$ be local coordinates on $Y=\partial
M$, and $(x,y)$ be the local coordinates on $M$ near $\partial M$; the metric $g$ in a
collar neighborhood $[0,\epsilon)_x\times \partial M$ takes
the form 
\begin{equation}\label{metric}
g=\frac{dx^2}{x^2}+\frac{h(x,y)}{x^2}=\frac{dx^2}{x^2}+\frac{\sum
h_{jk}(x,y)dy^jdy^k}{x^2},
\end{equation}
where $x\in C^{\infty}(M)$ is a boundary defining function for
$\partial M$ and $h$ is a smooth family of metrics on $Y=\partial M$.  In addition, if every geodesic
in $M$ reaches $\partial M$ both forwards and backwards, we say $M$ is
nontrapping.  The Poincar\'e disc $(\mathbb{B}^{n+1}, g)$ is a typical example of such manifold.  Indeed, considering the ball $\mathbb{B}^{n+1}=\{z\in\R^{n+1}: |z|<1\}$
endowed with the metric
\begin{equation}
g=\frac{4dz^2}{(1-|z|^2)^2},
\end{equation}
one can take $x=(1-|z|)(1+|z|)^{-1}$ as the boundary defining function and $\omega$ the coordinates on $\mathbb{S}^{n}$. Then the Poincar\'e metric 
takes the form $$g=\frac{dx^2}{x^2}+\frac{\frac14(1-x^2)^2 d\omega^2}{x^2},$$
where $d\omega^2$ is the standard metric on the sphere $\mathbb{S}^{n}$. 
\vspace{0.2cm}

Let $H=-\Delta_g-\frac{n^2}4$ where $\Delta_g$ is the Laplace-Beltrami operator on $(M^\circ,g)$. 
We recall the following result  about estimates for the spectral measure and the spectral projectors by Chen and Hassell in \cite[Theorem 1.6]{chenHassell}. We denote by $dE(\lambda)$ the spectral measure associated with the operator $H$, such that for every $F$ a Borel function on $\R$, we have 

$$
F(\sqrt H)=\int_{\R} F(\lambda)dE(\lambda)
$$
with domain 
$$
\left \{ \psi: \int_{\R}|F(\lambda)^2|\,d\langle \psi, E(\lambda)\psi \rangle  <\infty \right \}.  
$$
\begin{theorem}\label{specMeasure}
Let $n\ge 1$.
Suppose $(M^\circ,g)$ is a non trapping asymptotically hyperbolic manifold of dimension $n+1$. 
Assume that there 
has no pure point eigenvalue and has no resonance at the bottom of the continuous spectrum of  $H=-\Delta_g-\frac{n^2}4$.
 Then for some constant $C$, we have: for $\lambda \leq 1$
\begin{equation}\label{lowFreq}
\|dE(\lambda)\|_{p \to p'} \leq C \lambda^2,\,\,\, 1\leq p <2.
\end{equation}
For $\lambda \geq 1$, we have 
\begin{equation}\label{highFreq}
\|dE(\lambda)\|_{p \to p'} \leq \left \{
\begin{array}{c}
C \lambda^{(n+1)(1/p-1/p')-1},\,\,\, 1\leq p \leq \frac{2(n+2)}{n+4},\\
C \lambda^{n(1/p-1/2)},\,\,\,\, \frac{2(n+2)}{n+4} \leq p <2.\\
\end{array} \right.
\end{equation}

\end{theorem}

\begin{corollary}\label{specBound}
Suppose $(M^\circ,g)$ is a non trapping asymptotically hyperbolic manifold of dimension $n+1$ with $n\geq1$. 
Assume that there 
has no pure point eigenvalue and has no resonance at the bottom of the continuous spectrum of  $H$.
 Let $\delta\in (0, 1]$ and define the spectral projector on the frequencies $[k,k+\delta]$
\begin{equation}\label{c-kdelta}
\chi_{k}^\delta f=\int_0^\infty \chi_{[k,k+\delta]}(\lambda) dE(\lambda) f.
\end{equation}
Then for $k \geq 0$
\begin{equation}\label{highFreq2'}
\|\chi_{k}^{\delta}\|_{p \to p'} \leq 
C\delta (\delta+k)^2 (1+k)^{(n+1)(1/p-1/p')-3},\,\,\, 1\leq p \leq \frac{2(n+2)}{n+4}.
\end{equation}
In particular, we have
\begin{equation}\label{highFreq2}
\|\chi_{k}^{\delta}\|_{p \to p'} \leq 
C\delta (\delta+k)^{(n+1)(1/p-1/p')-1},\,\,\, \max\Big\{1,\frac{2(n+1)}{n+4}\Big\}\leq p \leq \frac{2(n+2)}{n+4}\ .
\end{equation}
\end{corollary}
\begin{proof}
By Theorem \ref{specMeasure} and integrating the spectral measure on $[k,k+\delta]$, one gets the desired result. 
\end{proof}

\begin{remark}
In the case of high frequencies $k \geq 1$, the previous theorem is known to hold on asymptotically conic manifold and even on manifolds with bounded geometry (see \cite{GHS}). In particular, asymptotically hyperbolic manifolds are examples of manifolds with bounded geometry in the following sense: suppose $(M,g)$ be a complete Riemannian manifold of dimension $n+1$ with $n\geq 2$. We assume that $M$ has {\sl $C^\infty$ bounded geometry}, i.e. the local injectivity radius $M$ has a positive lower bound $\varepsilon$,  the metric tensor $g_{ij}$ , expressed in normal coordinates in the ball of radius $\varepsilon/2$ around any point $z \in M$ is uniformly bounded in $C^\infty(B(z,\varepsilon/2))$ as $z$ ranges over $M$; and the inverse metric $g^{ij}$ is uniformly bounded in supremum norm. 
\end{remark}

The previous corollary has the following applications. 

\begin{theorem} [Reversed local-in-time Strichartz estimates] \label{strthm} Suppose $(M^\circ,g)$ is a non trapping asymptotically hyperbolic manifold of dimension $n+1$ with $n\geq1$. 
Assume that there 
has no pure point eigenvalue and has no resonance at the bottom of the continuous spectrum of  $H$.
Let $u$ be the unique solution of
\begin{equation}\label{8.1}
\begin{cases}
\bigl(\partial_t^2+H \bigr)u=0,\\
u|_{t=0}=f, \quad \partial_tu|_{t=0}=0.
\end{cases}
\end{equation}
Then for $\frac{2(n+2)}{n}\leq q\leq \infty$, $2\leq p<\infty$ and $s=(n+1)(\frac12-\frac1q)-\frac1p$
\begin{equation}\label{8.2}
\|u\|_{L^{q}(M^\circ;L^p([-1,1]))}\le C \|(Id+H)^{\frac s2}f\|_{L^2(M^\circ)}.
\end{equation}
\end{theorem}

\begin{theorem} [Reversed global-in-time Strichartz estimates] \label{strthm-g} 
Assume  $(M^\circ,g)$  and $H$ be in Theorem \ref{strthm}. 
Let $q_n=+\infty$ when $n=1,2$ and $q_n=\frac{2(n+1)}{n-2}$ when $n\geq3$.
For $\frac{2(n+2)}{n}\leq q\leq q_n$, $2\leq p<\infty$ and $s=(n+1)(\frac12-\frac1q)-\frac1p$, the global-in-time estimates hold
\begin{equation}\label{stri-g}
\|u\|_{L^{q}(M^\circ;L^p(\R))}\le C \|H^{\frac s2}f\|_{L^2(M^\circ)}.
\end{equation}
Furthermore, if $q_n<q\leq +\infty$ with $n\geq 3$, we have the global-in-time estimates 
\begin{equation}\label{stri-g'}
\|u\|_{L^{q}(M^\circ;L^p(\R))}\le C \|
  (H+ Id)^{s/2}
f\|_{L^2(M^\circ)}.
\end{equation}
\end{theorem}

The standard Strichartz estimates were established in \cite{SSWZ} but with an arbitrary small $\epsilon$-loss of regularity due to the lack of Littlewood-Paley theory in the non-doubling setting. Compared with the standard Strichartz estimates, these ones reverse the order of space-time integration with no loss.
Estimates in this form have been extensively studied in \cite{smith, ss} for waves on manifolds, \cite{becgold} for waves with potentials. 

Another application is a sharp maximal function estimate. In \cite{RV}, Rogers and Villarroya proved the following sharp maximal estimate for the wave operator in $\mathbb R^d$:

\begin{equation}
\big\|\sup_{t\in\R}|e^{it\sqrt{-\Delta}}f(x)|\big\|_{L^q(\R^d)}\lesssim \|f\|_{H^s(\R^d)}
\end{equation}
provided $q\geq\frac{2(d+1)}{d-1}$ and $s>d(\frac12-\frac1q)$. We get an analogous statement in our context:

\begin{corollary} [Maximal estimate] \label{max} Let  the manifold $(M^\circ,g)$ and  the operator $H$ be in Theorem \ref{strthm}.
Then the following holds
\begin{equation}
\big\|\sup_{t\in\R}|e^{it\sqrt{H}}f(x)|\big\|_{L^q(M^\circ)}\lesssim \|f\|_{H^s(M^\circ)}
\end{equation}
provided $q\geq\frac{2(n+2)}{n}$ and $s>(n+1)(\frac12-\frac1q)$.

\end{corollary}

\begin{proof}
This is just a consequence of Theorem \ref{strthm-g} with $p=2$
 and the Sobolev embedding $\dot H^{\frac12+\e}(\R)\hookrightarrow L^{\infty}_t(\R)$ for any $\e>0$.
\end{proof}

{\bf Acknowledgments:}\quad   Y. S. was partially supported by the Simons foundation. 
C. D. S. was supported by the NSF and the Simons foundation.
C. Wang was supported in part by NSFC 11971428 and National Support Program for Young Top-Notch Talents.
 J. Zhang was supported by  NSFC Grants (11771041, 11831004). \vspace{0.2cm}

\section{ Proof of Theorem ~\ref{strthm}}

 In \cite[Theorem 2.1]{blpcritical}, the authors show that the bounds \eqref{8.2} follows from Theorem \ref{specBound}, and their proof works equally well in our circumstances. Nonetheless, we include a proof for the sake of completeness which will serve as a model for certain global Strichartz
estimates.

To prove \eqref{8.2}, it suffices to show
that
\begin{equation}\label{8.4}
\bigl\|e^{it\sqrt{H}}f\bigr\|_{L^{q}(M^\circ;L^p([-1,1]))}\le C \|(I+H)^{\frac s2}f\|_{L^2(M^\circ)}.
\end{equation}
There exists a $\rho\in {\mathcal S}(\R)$ satisfying
$\text{supp }\Hat \rho\subset (-2,2)$ such that
\begin{equation}
\bigl\|e^{it\sqrt{H}}f\bigr\|_{L^{q}(M^\circ;L^p([-1,1]))} \lesssim \bigl\|\rho(t)e^{it\sqrt{H}}f\|_{L^{q}(M^\circ;L^p(\R))}.
\end{equation}
Then to prove \eqref{8.4}, it suffices to prove
\begin{equation}
\bigl\|\rho(t)e^{it\sqrt{H}}f\bigr\|_{L^{q}(M^\circ;L^p(\R))}\le C \|(I+H)^{\frac s2}f\|_{L^2(M^\circ)}.
\end{equation}

To prove this, we shall change notation a bit, and in particular $\delta=1$ in \eqref{c-kdelta}, let
\begin{equation}
\chi_k f=\int_0^\infty \chi_{[k,k+1]}(\lambda) dE(\lambda) f
\end{equation}
so that
$f=\sum_{k=0}^\infty \chi_kf$.  Then, for $\frac{2(n+2)}{n}\leq q\leq \infty$,  \eqref{highFreq2} yields
\begin{equation}\label{8.5}
\|\chi_kf\|_{L^{q}(M^\circ)}\lesssim (1+k)^{(n+1)(\frac12-\frac1q)-\frac12}\|f\|_{L^2(M^\circ)}, \quad k=0,1,2,\dots .
\end{equation}

To use this, we first note that by Sobolev estimates
$$\bigl\|\rho(t)e^{it\sqrt{H}}f\bigr\|_{L^{q}(M^\circ;L^p(\R))}\lesssim
\bigl\|\, |D_t|^{1/2-1/p}\bigl(\rho(t)e^{it\sqrt{H}}f\bigr)\bigr\|_{L^{q}_xL^2_t(\R\times M^\circ)}.$$
Let
$$F(t,x)= |D_t|^{1/2-1/p}\bigl(\rho(t)e^{it\sqrt{H}}f(x)\bigr)$$
denote the function inside the mixed-norm in the right, then
$$F(t,x)=\sum_{k=0}^\infty F_k(t,x),$$
where
$$F_k(t,x)= |D_t|^{1/2-1/p}\bigl(\rho(t)e^{it\sqrt{H}}\chi_kf(x)\bigr).$$
Consequently, its $t$-Fourier transform is
\begin{equation}\label{8.6}\Hat F_k(\tau,x)=|\tau|^{1/2-1/p}\int_0^\infty \Hat \rho(\tau-\la) \chi_{[k,k+1]}(\lambda) dE(\lambda) f(x).\end{equation}
Since $\text{supp }\Hat \rho \subset (-2,2)$, we conclude that
$$\int_{-\infty}^\infty F_k(t,x)\, \overline{F_\ell(t,x)}\, dt =(2\pi)^{-1}\int_{-\infty}^\infty
\Hat F_k(\tau,x)\, \overline{\Hat F_\ell(\tau,x)} \, d\tau =0
\quad \text{when } \, |k-\ell|>100.$$
As a consequence, we obtain
\begin{multline*}
\bigl(\, \int_{-\infty}^\infty \, \bigl| \, |D_t|^{1/2-1/p}\bigl(\rho(t)e^{it\sqrt{H}}f(x)\bigr)\, \bigr|^2 \, dt\, \bigr)^{1/2}
\\
\lesssim \bigl(\, \int_{-\infty}^\infty \sum_{k=0}^\infty |F_k(t,x)|^2 \, dt\, \bigr)^{1/2}
=(2\pi)^{-1/2} \bigl(\int_{-\infty}^\infty \sum_{k=0}^\infty |\Hat F_k(\tau,x)|^2 \, d\tau \, \bigr)^{1/2}.
\end{multline*}
Also, since $q\geq 2$, we conclude that 
$$\bigl\|\, |D_t|^{1/2-1/p}\bigl(\rho(t)e^{it\sqrt{H}}f\bigr)\bigr\|^2_{L^{q}_xL^2_t(\R\times M^\circ)}\lesssim \sum_{k=0}^\infty\int_{-\infty}^\infty \|\Hat F_k(\tau,x)\|_{L^{q}(M^\circ)}^2 \, d\tau.$$

Recalling \eqref{8.6}, the support properties of $\Hat \rho$, we see that 
\begin{align*}
&\bigl\|\, |D_t|^{1/2-1/p}\bigl(\rho(t)e^{it\sqrt{H}}f\bigr)\bigr\|^2_{L^{q}_xL^2_t(\R\times M^\circ)}\\&=
 \sum_{k=0}^\infty \int_{-\infty}^\infty |\tau|^{1-2/p} \, \Bigl\|\int_0^\infty \Hat \rho(\tau-\la) \chi_{[k,k+1]}(\lambda) dE(\lambda) f(x)\Bigr\|_{L^{q}(M^\circ)}^2 \, d\tau
\\
&=  \sum_{k=0}^\infty \int_{k-10}^{k+10} |\tau|^{1-2/p} \,
\Bigl(\int_k^{k+1}\lambda^{(n+1)(\frac12-\frac1q)-\frac12} d\lambda\, \|\chi_k f(x)\|_{L^2(M^\circ)}\Bigr)^2 \, d\tau \, 
\\
&\lesssim \, \sum_{k=0}^\infty (1+k)^{1-2/p}(1+k)^{(n+1)(1-\frac2q)-1}\|\chi_kf\|_{L^2(M^\circ)}^2
\\
&=\bigl(\, \sum_{k=0}^\infty\|(1+k)^{(n+1)(\frac12-\frac1q)-\frac1p}\chi_kf\|_{L^2(M^\circ)}^2\, \bigr)^{1/2}\approx \|(I+H)^{s/2}f\|_{L^2(M^\circ)}^2,
\end{align*}
as desired, which completes the proof. 

\section{The proof of Theorem \ref{strthm-g}}

In this section, we prove the global result about Theorem \ref{strthm-g} which is direct consequence of 
the following Proposition.

\begin{prop} Suppose the operator $H$, $q_n$ and the manifolds $(M^\circ,g)$ to be in Theorem \ref{strthm-g}. 
For $\frac{2(n+2)}{n}\leq q\leq q_n$, $2\leq p<\infty$ and $s=(n+1)(\frac12-\frac1q)-\frac1p$, the global-in-time estimates hold
\begin{equation}\label{stri-g}
\|u\|_{L^{q}(M^\circ;L^p(\R))}\le C \|H^{\frac s2}f\|_{L^2(M^\circ)}.
\end{equation}
Furthermore, if $q_n<q\leq +\infty$ with $n\geq 3$, we have the global-in-time estimates 
\begin{equation}\label{stri-g'}
\|u\|_{L^{q}(M^\circ;L^p(\R))}\le C \|
H^{3/4-1/(2p)}
  (H+ I)^{s/2+1/(2p)-3/4}
f\|_{L^2(M^\circ)}.
\end{equation}
\end{prop}

\begin{proof}
Consider first the case with $q\leq q_n$.
To prove \eqref{stri-g}, 
 It suffices to show that, for $0<\e<1$, there is a uniform constant $C$ independent of $\e$ so that
 \begin{equation}\label{r.17'}\tag{9.17$'$}
 \|e^{it\sqrt{H}}f\|_{L^{q}(M^\circ;L^p([-\frac1\e,\frac1\e]))}
 \le C\|(\sqrt H+\e I)^{s}f\|_{L^2(M^\circ)}.
 \end{equation}
  To this end, similar to the proof of Theorem~\ref{strthm}, it suffices to show that
 if we fix $\rho\in {\mathcal S}(\R)$ with $\text{supp }\Hat \rho\subset (-2,2)$, then we have
 the uniform bounds
 \begin{equation}
 \|\rho(\e t)\, e^{it\sqrt{H}}f\|_{L^{q}(M^\circ;L^p(\R))} \le C\|(H+\e I)^{s/2}f\|_{L^2(M^\circ)}.
 \end{equation}
 As before, we use $\dot H^{\frac12-\frac1p}_t(\R)\to L^{p}_t(\R)$ Sobolev estimates  to deduce that
 $$\|\rho(\e t)\, e^{it\sqrt{H}}f\|_{L^{q}(M^\circ;L^p(\R))} \lesssim
 \| \, |D_t|^{1/2-1/p} (\rho(\e t)e^{it\sqrt{H}}f)\, \|_{L^{q}(M^\circ;L^2(\R))}.
 $$
 Similarly as before, let
 $$F^\e(t,x)=|D_t|^{1/2-1/p}\bigl(\rho(\e t)e^{it\sqrt{H}}f\bigr).$$
 If we take the Fourier transform in $t$, we deduce that
 $$\widehat {F^\e}(\tau,x)=|\tau|^{1/2-1/p} \, \e^{-1} \, \bigl(\Hat \rho(\e^{-1}(\tau-\sqrt{H})) f\bigr)(x)
 =\sum_{k=0}^\infty \widehat {F^\e_k}(\tau,x),$$
 where
 $$ \widehat {F^\e_k}(\tau,x)=|\tau|^{1/2-1/p} \, \e^{-1} \, \bigl(\Hat \rho(\e^{-1}(\tau-\sqrt{H}))\circ\chi^\e_k f\bigr)(x)$$
 and $\chi^\e_{k\varepsilon}$ is the spectral projection operator for $I_k$ associated with $\sqrt{H}$ which is given by 
 \begin{equation}
\chi^\e_{k\varepsilon} f=\int_0^\infty \chi_{I_k}(\lambda) dE(\lambda) f, \quad I_k=[k\e, (k+1)\e), \quad k=0, 1,2,3,\dots.
\end{equation}
Since $\text{supp }\Hat \rho\subset (-2,2)$, one sees the fact that
 $\Hat \rho(\e^{-1}(\tau-\sqrt{H}))\circ\chi^\e_k$ vanishes if $\tau\notin [\e(k-100),\e(k+100)]$.  Consequently,
 $$\int_{-\infty}^\infty
 F^\e_k(t,x) \, \overline{F^\e_\ell(t,x)} \, dt
 =(2\pi)^{-1}\int_{-\infty}^\infty \widehat {F^\e_k}(\tau,x)\, \overline{\widehat {F^\e_\ell}(\tau,x)} \, d\tau
 =0 \quad \text{if } \, \, |k-\ell|>100.$$
 Since $f=\sum_{k=0}^\infty\chi_{k\e}^\e f$, as a consequence, it gives 
 \begin{align*}\int_{-\infty}^\infty |F^\e(t,x)|^2 \, dt
 =\int_{-\infty}^\infty \bigl| \, \sum_{k=0}^\infty F^\e_k(t,x)\, \bigr|^2 \, dt
& \lesssim \sum_{k=0}^\infty \int_{-\infty}^\infty |F^\e_k(t,x)|^2 \, dt
\\
&\lesssim  \sum_{k=0}^\infty \int_{-\infty}^\infty |\widehat {F^\e_k}(\tau,x)|^2 \, d\tau.
 \end{align*}
 From the above, we deduce that
 \begin{align*}
 \|\rho(\e t)\, e^{it\sqrt{H}}f\|^2_{L^{q}(M^\circ;L^p(\R))} &\lesssim \sum_{k=0}^\infty \int_{-\infty}^\infty
  \| \widehat {F^\e_k}(\tau,\, \cdot \, )\|^2_{L^{q}(M^\circ)} \,
  d\tau
  \\
  & =\e^{-2} \sum_{k=0}^\infty \int_{(k-10)\e}^{(k+10)\e} |\tau|^{1-2/p}
  \|\Hat \rho(\e^{-1}(\tau-\sqrt{H}))\chi^\e_{k\e}f\|_{L^q}^2 \, d\tau
  \\
  &\lesssim \e^{-2}\sum_{k=0}^\infty \e\cdot ((k+10)\e)^{1-2/p}\|\chi^\e_{k\varepsilon} f \,    \|_{L^q}^2.
   \end{align*}
When $q\leq q_n$,
it follows
 from \eqref{highFreq2} of Corollary~\ref{specBound} that
 \begin{equation}\label{r.19}
 \|\chi^\e_{k\e}f\|_{L^{q}(M^\circ)}\lesssim \e^{1/2} \, ((k+1)\e)^{(n+1)(\frac12-\frac1q)-\frac12} \, \|\chi^\e_{k\e}f\|_{L^2(M^\circ)},
 \quad k=0,1,2,3,\dots.
 \end{equation}
 therefore, recall $s=(n+1)(\frac12-\frac1q)-\frac1p$, we further obtain
  \begin{align*}
 \|\rho(\e t)\, e^{it\sqrt{H}}f\|^2_{L^{q}(M^\circ;L^p(\R))} &\lesssim 
  \e^{-2}\sum_{k=0}^\infty \e \cdot ((k+1)\e)^{1-2/p}\, \bigl(\e^{1/2}((k+1)\e )^{(n+1)(\frac12-\frac1q)-\frac12}\bigr)^2
  \|\chi^\e_{k\e}f\|_2^2
  \\
  &=\e^{-2}\sum_{k=0}^\infty \e^2 \, ((k+1)\e)^{2[(n+1)(\frac12-\frac1q)-\frac1p]}\, \|\chi^\e_{k\e}f\|_2^2
  \\
  &=
  \sum_{k=0}^\infty \bigl\| \, ((k+1)\e)^{s}\chi^\e_{k\e}f \, \bigr\|_2^2 \approx
  \bigl\| \, (\sqrt H+\e I)^{s}f\, \bigr\|_2^2,
 \end{align*}
 as desired.

 While for $q\geq q_n$,
we use \eqref{highFreq2'} instead of \eqref{highFreq2} in Corollary~\ref{specBound} to obtain
 \begin{equation}\label{r.19}
 \|\chi^\e_{k\e}f\|_{L^{q}(M^\circ)}\lesssim \e^{1/2} \, (k+1)\e (k\e+1)^{(n+1)(\frac12-\frac1q)-\frac 32}  \, \|\chi^\e_{k\e}f\|_{L^2(M^\circ)},
 \quad k=0,1,2,3,\dots.
 \end{equation}
Therefore, we further obtain
  \begin{align*}
 \|\rho(\e t)\, e^{it\sqrt{H}}f\|^2_{L^{q}(M^\circ;L^p(\R))} &\lesssim 
  \e^{-2}\sum_{k=0}^\infty \e \cdot ((k+1)\e)^{1-2/p}\, \bigl(\e^{1/2}
  (k+1)\e
  (k\e +1)^{(n+1)(\frac12-\frac1q)-\frac32}\bigr)^2
  \|\chi^\e_{k\e}f\|_2^2
  \\
  &=\e^{-2}\sum_{k=0}^\infty \e^2 \,
(  (k+1)\e)^{3-2/p}
   (k\e+1)^{2[(n+1)(\frac12-\frac1q)-\frac32]}\, \|\chi^\e_{k\e}f\|_2^2
  \\
  &=
  \sum_{k=0}^\infty \bigl\| \, ((k+1)\e)^{3/2-1/p}(k\e+1)^{(n+1)(\frac12-\frac1q)-\frac32}\chi^\e_{k\e}f \, \bigr\|_2^2\\& \approx
  \bigl\| \, (\sqrt H+\e I)^{3/2-1/p}
  (\sqrt H+ I)^{(n+1)(\frac12-\frac1q)-\frac32}
  f\, \bigr\|_2^2,
 \end{align*}
which completes the proof.

\end{proof}

\section{Small data well-posedness \\ and Strauss conjecture on asymptotically hyperbolic manifolds}

We now draw some consequences of the previous estimates. We provide an application of the global reversed Strichartz estimates established above about global existence for nonlinear waves. Consider the wave equation  with a nonlinearity satisfying
\begin{equation*}\label{eq-Fp}
|F_p(u)|+|u||F'_p(u)|\le C |u|^p,
\end{equation*}
for some constant $C>0$, 
\begin{equation}\label{NKG}
\begin{cases}
\partial_{t}^2u(t,z)+Hu(t,z)=F_p(u), \\ u(0)=u_0(z), ~\partial_tu(0)=u_1(z).
\end{cases}
\end{equation}

In \cite{SSWZ}, the authors adressed the small data well-posedness for any power $p\in (1,1+\frac{4}{n})$, leaving the end-point case open. This issue was raised by the methods we used which was not allowing us to get the suitable Strichartz estimates for $p=1+\frac{4}{n}$. We address here the latter and then focus only on $p=1+\frac{4}{n}$. We prove

\begin{theorem}\label{GWP}
Let  $(M^\circ,g)$ be a non-trapping asymptotically hyperbolic manifold of dimension
$n+1$. 
Assume that there 
has no pure point eigenvalue and has no resonance at the bottom of the continuous spectrum of  $H$.
Then there exists a constant $\nu_1>0$ such that  the Cauchy problem
\begin{equation}\label{eqNLW}
\begin{cases}
\partial_{t}^2u+H u=F_{1+\frac{4}{n}}(u), \quad (t,z)\in I\times M^\circ; \\ u(0)=\nu u_0(z),
~\partial_tu(0)=\nu u_1(z),
\end{cases}
\end{equation}
 has  a global solution, provided $|\nu|\le\nu_1 $ and 
 $$ \|H^{\frac 14}u_0\|_{L^2(M^\circ)}+ \|H^{- \frac 14}u_1\|_{L^2(M^\circ)} \le 1.
$$
\end{theorem}
\begin{proof}
Consider the complete space $L^q(M^\circ;L^p_t(\R))$ for $q=p=\frac{2(n+2)}{n}$, i.e. $L^{\frac{2(n+2)}{n}}(\mathbb R \times M^\circ)$. Define the map $\mathcal T$ by $v=\mathcal T u$ where $v$ solves, given $u \in L^{\frac{2(n+2)}{n}}(\mathbb R \times M^\circ)$, 
\begin{equation}
\begin{cases}
\partial_{t}^2v +H v=F_{1+\frac{4}{n}}(u), \quad (t,z)\in \R \times M^\circ; \\ v(0)=\nu u_0(z),
~\partial_t v(0)=\nu u_1(z).
\end{cases}
\end{equation}
Estimate \eqref{stri-g} then gives  
 \begin{equation*}
\begin{split}
&\|v(t,z)\|_{L^{2(n+2)/n}(\R \times M^\circ)}
\lesssim \nu \Big (\|H^{\frac 14}u_0\|_{L^2(M^\circ)}+ \|H^{-\frac 14}u_1\|_{L^2(M^\circ)}\Big )+ \|u\|^{1+4/n}_{L^{2(n+2)/(n+4)}(\R\times M^\circ)}.
\end{split}
\end{equation*}

A standard computation shows that if $\nu$ is small enough, $\mathcal T$ maps a ball of $L^{p+1}(\mathbb R^+ \times M^\circ)$ into itself and is actually a contraction, hence by the Banach fixed point theorem this leads to the desired result (see for instance \cite{SSW} for more details). 
\end{proof}

\providecommand{\MR}[1]{}
\providecommand{\bysame}{\leavevmode\hbox to3em{\hrulefill}\thinspace}
\providecommand{\MR}{\relax\ifhmode\unskip\space\fi MR }
\providecommand{\MRhref}[2]{%
  \href{http://www.ams.org/mathscinet-getitem?mr=#1}{#2}
}
\providecommand{\href}[2]{#2}


\begin{thebibliography}{1}

\bibitem{blpcritical}
Nicolas Burq, Gilles Lebeau, and Fabrice Planchon, \emph{{Global existence for
  energy critical waves in 3-{D} domains}}, J. Amer. Math. Soc. \textbf{21}
  (2008), no.~3, 831--845. \MR{2393429 (2009f:35225)}

\bibitem{becgold}
M. Beceanu and M. Goldberg.
Strichartz estimates and maximal operators for the wave equation in $\R^3$
\newblock{\em J. Funct. Anal.}  266 (3) 1476-1510 (2014).

\bibitem{Chen} X. Chen, Resolvent and spectral measure on non-trapping asymptotically hyperbolic manifolds III: 
Global-in-time Strichartz estimates without loss, \newblock{\em Ann. I. H. Poincar\'e}  {\bf 35} (2018), 803-829.

\bibitem{CH1} X. Chen and A. Hassell, Resolvent and spectral measure on non-trapping asymptotically hyperbolic manifolds I: 
Resolvent construction at high energy, \newblock{\em Comm. PDE}  {\bf 41} (2016), 515-578.

\bibitem{chenHassell}
Xi~Chen and Andrew Hassell, \emph{Resolvent and spectral measure on
  non-trapping asymptotically hyperbolic manifolds {II}: {S}pectral measure,
  restriction theorem, spectral multipliers}, Ann. Inst. Fourier (Grenoble)
  \textbf{68} (2018), no.~3, 1011--1075. \MR{3805767}

\bibitem{GHS}
Colin Guillarmou, Andrew Hassell, and Adam Sikora, \emph{Restriction and
  spectral multiplier theorems on asymptotically conic manifolds}, Anal. PDE
  \textbf{6} (2013), no.~4, 893--950. \MR{3092733}
  
  \bibitem{M} R. Mazzeo, The Hodge cohomology of a conformally compact metric, J. Diff. Geom. 28(1988), 309-339.

\bibitem{MM} R. Mazzeo, R. B. Melrose, Meromorphic extention of the resolvent on complete spaces
with asymptotically constant negative curvature, J. Func. Anal. 75(1987), 260-310.

  
  \bibitem{RV} K.M. Rogers, P. Villarroya, \emph{Sharp estimates for maximal operators associated to the wave equation}, Ark. Mat. \textbf{46}
(2008) 143-151.

\bibitem{smith}
 H. F. Smith
\newblock Spectral cluster estimates for $C^{1,1}$ metrics.
\newblock {\em Amer. J. Math.}, 128(2006), 1069-1103.

\bibitem{ss}
 H. F. Smith and C. D. Sogge
\newblock On the $L^p$ norm of spectral clusters for compact manifolds with boundary.
\newblock {\em Acta Math.}, 198 (2007), 107-153.

\bibitem{SSW}
Yannick Sire, Christopher~D. Sogge, and Chengbo Wang, \emph{The {S}trauss
  conjecture on negatively curved backgrounds}, Discrete Contin. Dyn. Syst.
  \textbf{39} (2019), no.~12, 7081--7099. \MR{4026182}

\bibitem{SSWZ}
Yannick Sire, Christopher~D. Sogge, Chengbo Wang, and Junyong Zhang,
  \emph{Strichartz estimates and {S}trauss conjecture on non-trapping
  asymptotically hyperbolic manifolds}, Trans. Amer. Math. Soc. \textbf{373}
  (2020), no.~11, 7639--7668. \MR{4169670}


\end{thebibliography}
\end{document}